\documentclass{article}

\usepackage{amsthm,amsmath,amssymb,dsfont,latexsym,enumerate}

\newtheorem{theorem}{Theorem}[section]

\newtheorem{lemma}[theorem]{Lemma}
\newtheorem{corollary}[theorem]{Corollary}
\newtheorem{remark}[theorem]{Remark}

\newtheorem{example}[theorem]{Example}
\newtheorem{assumption}[theorem]{Assumption}

\newcommand{\N}{\mathbb{N}}
\newcommand{\Z}{\mathbb{Z}}

\newcommand{\R}{\mathbb{R}}

\newcommand{\A}{\mathcal{A}}
\newcommand{\B}{\mathcal{B}}

\newcommand{\F}{\mathcal{F}}
\newcommand{\G}{\mathcal{G}}
\renewcommand{\H}{\mathcal{H}}
\newcommand{\X}{\mathcal{X}}
\newcommand{\Y}{\mathcal{Y}}
\newcommand{\M}{\mathcal{M}}

\renewcommand{\P}{\mathbb{P}}
\newcommand{\E}{\mathbb{E}}
\newcommand{\D}{\mathbb{D}}
\newcommand{\cov}{\mathrm{Cov}}
\newcommand{\law}[1]{\text{Law}(#1)}

\newcommand{\Pas}{\text{a.s.}}
\newcommand{\ind}{\mathds{1}}

\newcommand{\eps}{\varepsilon}
\newcommand{\la}{\lambda}
\newcommand{\ga}{\gamma}
\newcommand{\ka}{\kappa}
\renewcommand{\th}{\theta}
\newcommand{\dtv}{d_{\text{TV}}}

\newcommand{\dint}{\mathrm{d}}  

\usepackage[margin=30mm]{geometry}

\title{Functional Central Limit Theorem and Strong Law of Large Numbers for Stochastic Gradient Langevin Dynamics\thanks{
		The second author was supported by the National Research, Development and Innovation Office 
		within the framework of the Thematic Excellence Program 2021; National Research subprogramme 
		``Artificial intelligence, large networks, data security: mathematical foundation and applications''.}}

\author{A. Lovas \and M. R\'asonyi}

\date{\today}

\begin{document}

\maketitle

\begin{abstract}
	We study the mixing properties of an important optimization algorithm of machine learning: the stochastic gradient Langevin dynamics (SGLD) with a fixed step size. The data stream is not assumed to be independent hence the SGLD is not a Markov chain, merely a \emph{Markov chain in a random environment}, which complicates the mathematical treatment considerably. We derive a strong law of large numbers and a functional central limit theorem for SGLD.
\end{abstract}

\section{Introduction}

We consider a recursive stochastic scheme called ``stochastic gradient Langevin dynamics'' (SGLD), first suggested by Welling and Teh \cite{wt}. 
Let $\la>0$ be the \emph{stepsize}, the measurable function $H:\R^d\times\R^m\to\R^d$ the 
\emph{updating function} and define the $\R^d$-valued
stochastic process $\th_n$, $n\geq 1$ recursively by

\begin{equation}\label{eq:theta}
\th_{n+1}=\th_n-\la H(\th_n,Y_n)+\sqrt{2\la}\xi_{n+1}.
\end{equation}

Here $\xi_n$, $n\geq 1$ is an independent sequence of standard $d$-dimensional Gaussian random variables, $Y_n$, $n\in\Z$ is an $\R^m$-valued strict sense stationary process, independent of $(\xi_{n})_{n\in\N}$, which represents the data stream fed into this procedure. Furthermore, we assume (for simplicity) that the initial value $\th_{0}\in\R^d$ is deterministic. 

The algorithm \eqref{eq:theta} is used for approximate sampling from high-dimensional probability distributions that are not necessarily log-concave. More precisely, let $U:\R^d\to\R_+$ be differentiable with derivative $h=\nabla U$ such that $h(\th)=E[H(\th,Y_0)]$, $\th\in\R^{d}$. Assume $U$ has a unique minimum at $\th^{\dagger}$. For $\la$ small and $n$ large, $\law{\th_n}$ is expected to be close to the probability defined by
$$
\pi(A)=\frac{\int_A e^{-U(\th)}d\th}{\int_{\R^d} e^{-U(\th')}\dint\th'},\ A\in\B(\R^d),
$$
see e.g.\ \cite{wt,6,dependent}. If $\sqrt{2\la}$ in \eqref{eq:theta} is replaced by $\sqrt{2\la/\beta}$ for some $\beta>0$ then the procedure samples from a distribution with density proportional to $e^{-\beta U(x)}$ which means, for $\beta$ large, that 
\begin{equation}\label{mocca}
E[\th_{n}]\approx \int_{\R^{d}}x\,\pi(dx)\approx\th^{\dagger},
\end{equation}
for $n$ large enough and $\la$ small enough. (In this paper we keep $\beta=1$ for simplicity.) 
		
\begin{example}
{\rm We consider a regularized logistic regression where $m\geq 2$, $d:=m-1$ and $(Q_{n},Z_{n})\in\{0,1\}\times\R^m$, $n\in\Z$ is a stationary sequence of random variables. The purpose is to optimize the regression parameters $\theta\in\R^d$ in such a way that the functional
$$
U(\th):=-E\left[\ln[\sigma^{Q_{0}}(\langle \th,Z_{0}\rangle)(1-\sigma(\langle \th,Z_{0}\rangle))^{1-Q_{0}}]\right]+c|\th|^{2}
$$
is minimized, where $\sigma(x)=1/(1+e^{-x})$ is the sigmoid function and $c>0$ is a constant. One thus tries to guess the binary variable $Q$ from the variables $Z$. We then have
$$
H^{i}(\th,(q,z))=-(q-\sigma(\langle \th,Z_{0}\rangle))z^{i}+2c \th^{i}
$$
for all $i=1,\ldots,d$.
		
As can be easily verified, this functional satisfies Assumption \ref{as:oldpaper}. The SGLD algorithm in this context could be applied to standard sentiment analysis problems where, based on the occurrences of key words (represented by the coordinates of $Z$) it should be decided whether a given review on a webshop is positive or not ($Q=1$ or $Q=0$), see e.g.\ \cite{amazon}. 

Review data continuously arrive and often exhibit temporal dependencies and non-i.i.d. characteristics. This is because customers' reviews can be influenced by previous reviews, current trends, or the changing sentiment of other customers, leading to dependencies between reviews. Consequently, the occurrence of certain key words and the overall sentiment may not be independent across reviews. For such sentiment analysis problems, variants of stochastic gradient descent are commonly used. However, due to the lack of convexity, it is worth considering the use of SGLD. 

Furthermore, sentiment analysis faces the challenge of concept drift, which refers to the situation where the underlying sentiment distribution of the data changes over time. This could be due to various factors, such as changes in product features, external events, or trends. The SGLD algorithm is capable of adapting to concept drift scenarios by continuously updating the model parameters as new data arrives.} 

\end{example}

One would try to numerically approximate the integral in \eqref{mocca} by
$$
\frac{\th_{0}+\ldots+\th_{n-1}}{n}.
$$
However, to guarantee the consistency of such a procedure, one needs to establish a corresponding law of large numbers.
		
In the case where $\la$ in \eqref{eq:theta} is replaced by $\la_{n}$ with a decreasing sequence $\la_n$, $n\geq 0$, under suitable assumptions, the averages
\begin{equation}\label{ito}
\frac{\sum_{k=0}^{n-1}\lambda_k \phi (\theta_k)}{\sum_{k=0}^{n-1}\lambda_k}.
\end{equation}
converge almost surely to $\int_{\R^d}\phi(z)\pi(dz)$ for appropriate functions $\phi$ as shown in \cite{ttv}, where a related central limit theorem is also established. 
		
In the case of fixed $\la$, \cite{vzt} estimated the $L^2$ distance of the averages from the mean of $\pi$. Both these papers, like most available studies, assume that $Y_n$, $n\in\Z$ are i.i.d. This does not hold true in several applications, prominently in the case of financial times series, see e.g.\ \cite{laruelle}, where stochastic approximation schemes were treated in a setting with possibly dependent data. See also \cite{6,dependent,tikosi2021convergence,lovas} for more about SGLD with dependent data. 
		
When the $Y_{n}$ are independent, $\th_{n}$ is a Markov chain. However, the case of general stationary $Y_{n}$ is an order of magnitude more involved mathematically since $\th_{n}$ is only a Markov chain \emph{in a random environment}, see Section \ref{sec:main} for details. 

In this article, we establish a law of large numbers (LLN) for functionals for \eqref{ito} when employing a fixed stepsize $\lambda>0$. Additionally, we will establish an invariance principle. These results serve as crucial theoretical guarantees for the consistency of estimates, such as \eqref{ito}, and form the foundation for constructing confidence intervals for these estimates. Our work builds upon and extends the findings in \cite{average}, where LLN and CLT were shown for the stochastic gradient method with dependent data, specifically in the special case of a linear updating rule.

Our arguments are based on results of \cite{herrndorf1984} which require establishing mixing properties for the process $\theta_{t}$. The recent paper \cite{Truquet1} is closely related to this part of our work: it shows mixing for a certain class of processes. That setting, however, does not cover ours since the strong minorization property $\mathbf{A2}$ in \cite{Truquet1} does not hold for our processes. 

Section \ref{sec:main} states and explains our main results. Their proof in Section \ref{sec:proof} is presented in a series of subsections.

\section{The main result}\label{sec:main}

First we formulate our working assumption on the stochastic iterative scheme given by \eqref{eq:theta}.
\begin{assumption}\label{as:oldpaper} 
	There are $\Delta,b>0$ such that, for all $\th\in\R^d$ and $y\in\R^m$,
	\begin{equation}\label{eq:dis}
	\langle H(\th,y),\th\rangle\geq \Delta \lVert\th\rVert^2-b,
	\end{equation}
	and for some $K>\Delta/\sqrt{2}$, 
	\begin{equation}\label{eq:linear}
	\lVert H(\th,y)\rVert\leq K( \lVert\th\rVert+\lVert y\rVert+1).
	\end{equation}  
	Furthermore, we assume that the process $(Y_t)_{t\in\Z}$ is strictly stationary, and there is $M>0$ such that 
	\begin{equation}\label{eq:bound}
	\lVert Y_{0}\rVert\leq M\ \Pas
	\end{equation}
\end{assumption}

Condition \eqref{eq:dis} is a standard dissipativity requirement, \eqref{eq:linear} is also mild and holds for Lipschitz-continuous $H$. By stationarity, \eqref{eq:bound} implies uniform boundedness of the data stream. This may look stringent from the mathematical
point of view, but it is evidently applicable in practice due to two main reasons. First, many real-world applications involve data that can be naturally bounded within certain ranges. For example, pixel values in images are confined to specific ranges (e.g., 0 to 255 for grayscale images). Second, scaling the data to a compact domain is a common preprocessing step in machine learning. In conclusion, the assumptions we have made are met by a wide range of learning problems of considerable practical importance.

Next, we briefly recall the main concepts of $\alpha$-mixing. Throughout this paper the probability space is $(\Omega, \F, \P)$, and 
for any two sub-$\sigma$-algebras $\G,\H\subset\F$, we define the measure of dependence
\begin{equation}\label{eq:dep}
\alpha (\G,\H) = \sup\limits_{G\in\G, H\in\H} \left|\P (G\cap H)-\P (G)\P (H)\right|.
\end{equation}	
Furthermore, for an arbitrary sequence of random variables $(W_t)_{t\in\Z}$, we define 
the $\sigma$-algebras $\F_{t,s}^W:=\sigma \left(W_k,\,t\le k\le s\right)$, $-\infty\le t\le s\le\infty$,
and introduce the dependence coefficients 
$$\alpha_j^W (n) = \alpha\left(\F_{-\infty,j}^W,\F_{j+n,\infty}^W\right),\,\,j\in\Z.$$ 
The mixing coefficient of $W$ is $\alpha^W (n)= \sup_{j\in\Z}\alpha_j^W (n)$, $n\ge 1$ 
which is obviously non-increasing in $n$. Note that, for strictly stationary $W$, $\alpha_j^W (n)$ 
does not depend on $j$, and thus $\alpha^W (n)=\alpha_0^W (n)$. We say that $W$ is $\alpha$-mixing if $\lim_{n\to\infty}\alpha^W (n)=0$.

\begin{assumption}\label{as:mixing} For some $\epsilon>0$,
	the $\alpha$-mixing coefficients $\alpha^{Y}(n)$, $n\in\N$ 
	satisfy 
	$$
	\sum_{n=1}^{\infty}\alpha^{Y}(n)^{1-\epsilon}<\infty.
	$$	
\end{assumption}

In \cite{lovas} it was established (under somewhat weaker conditions than Assumption \ref{as:oldpaper}) that $\law{\th_{n}}$ converges in total variation to a limiting probability $\mu_{\la}$ as $n\to\infty$. A rate estimate of the order $\exp(-n^{1/3})$
was obtained. Clearly, $\mu_{\la}$ differs from $\pi$ and the bias is $O(\sqrt{\la})$ under suitable conditions, see \cite{dependent}.

In this paper, using results of \cite{gerencser2022invariant}, we prove an exponential convergence rate of $\law{\th_{n}}$ to $\mu_{\la}$ under Assumption \ref{as:oldpaper}. More importantly, a functional central limit theorem is established under the additional Assumption \ref{as:mixing}. In the sequel, $\phi:\R^d\to\R$ denotes an at most polynomially growing measurable function i.e. for fixed but arbitrary constants $c_\phi,r>0$,
\begin{equation}\label{eq:polygrowth}
|\phi (\th)|\le c_\phi (1+\lVert \th\rVert^r),\,\,\th\in\R^{d}.
\end{equation} 
Our main results are summarized in the next two theorems.

\begin{theorem}\label{thm:LLN}
Let Assumption \ref{as:oldpaper} be in force, and $0<\la\leq \frac{\Delta}{K^2}$ be fixed. Then there is a strictly stationary process $(\th_t^\ast)_{t\in\N}$ on $\R^d$ and there are constants $c,\ka>0$ depending only on $\la$, $\Delta$, $b$, $K$ and $M$ such that for any $k\in\N$ and indices $0\le i_1<\ldots<i_k$,
$$
\dtv \left(\law{(\th_{i_1+n},\ldots,\th_{i_k+n})},\law{(\th_{i_1}^\ast,\ldots,\th_{i_k}^\ast)}\right)
\le c e^{-\ka n}.
$$
Furthermore, we have
$$
\frac{1}{n}{\sum_{j=1}^{n}\phi(\th_{j})}\to \E (\phi (\th_0^\ast)),\,\, n\to\infty
$$
almost surely and in $L^{p}$, for all $p\geq 1$ provided that $(Y_t)_{t\in\Z}$ is ergodic.
\end{theorem}

\begin{theorem}\label{thm:CLT}
Under Assumptions \ref{as:oldpaper} and \ref{as:mixing}, for $0<\la\le\frac{\Delta}{K^2}$,
the process $X_n:=\phi (\th_n)-\E (\phi (\th_n))$, $n\in\N$ satisfies the invariance principle i.e. for $S_n=X_1+\ldots+X_n$, $\E S_n^2/n\to\sigma^2$ for some $\sigma\geq 0$ and the sequence of random functions
$$
B_n(t) = \frac{S_{\lfloor nt\rfloor}}{\sqrt{n}},\,\,t\in [0,1],\,n\ge 1
$$
is weakly convergent to $\sigma B_{t}$, $t\in [0,1]$ on $D[0,1]$ (the Skorokhod space endowed with the Skorohod topology) as $n\to\infty$. Here $B_{t}$, $t\in [0,1]$ is a standard Brownian motion.
\end{theorem}	

\begin{remark}
{\rm As we shall see later (c.f. Corollary \ref{cor:drift} and Lemma \ref{lem:annealed}), it is also true that $\E (\phi(\th_n))\to\E (\phi(\th_0^\ast))$ exponentially fast as $n\to\infty$ hence the biased sequence $X_n':=\phi (\th_n)-\E (\phi (\th_0^\ast))$ also satisfies the invariance principle.
}
\end{remark}

\section{Proofs}\label{sec:proof}

Throughout the rest of the paper, we use the 
notation $\X := \R^d$ and $\Y:=\R^m$ moreover $\B(\X)$ will be used for the standard Borel $\sigma$-algebra of $\X$. 
As pointed out in Section 5 of \cite{lovas} and also in \cite{tikosi2021convergence}, the recursive 
stochastic scheme \eqref{eq:theta} can be considered as a Markov chain in an exogenous random environment (MCRE) 
which means that there is a parametric kernel\footnote{That is, $Q(\cdot,\cdot,A)$ is measurable
for all $A\in\mathcal{B}(\mathcal{X})$ and $Q(y,x,\cdot)$ is a probability for all $(y,x)\in\mathcal{Y}\times\mathcal{X}$.}  
$Q:\Y\times\X\times\B(\X)\to [0,1]$ such that
$$
\P (\th_{t+1}\in A\mid (\th_i)_{0\le i\le t},\,(Y_j)_{j\in\Z}) = Q(Y_t,\th_t,A)
$$
almost surely, for all $A\in\mathcal{B}(\mathcal{X})$.
In our case, transition kernel is given by
\begin{equation*}
Q(y,\theta,A) = \P\left(\th-\la H(\th,y)+\sqrt{2\la}\xi_0  \in A\right),\,\,y\in\R^m,\th\in\R^d,A\in\B(\X),
\end{equation*}
where $\xi_0$ is as in the recursion \eqref{eq:theta} i.e. a standard $d$-dimensional Gaussian random variable.

Here we give a brief explanation of the proof strategy. First, we fix a trajectory of $Y$ (that is,
we consider the ``quenched'' version of the process) and using a standard representation of MCREs by iterated random functions, we deduce an upper estimate for the 
coupling probability between realizations of the chain starting from different, possibly random, initial values
Lemma \ref{lem:coupling}). 
To achieve this, we demonstrate that small sets, where coupling can occur with a positive probability, are visited frequently enough with large probability. The so-called ''annealed version'' of this crucial result 
(Lemma \ref{lem:annealed}) allows us to establish that the process $(\th_t)_{t\in\N}$ inherits the mixing properties of the environment (Lemma \ref{lem:mixing}). 
The proof of Theorem \ref{thm:LLN} also heavily relies on this inequality. We actually 
prove a bit more: we show that there exist an almost surely finite random time at which suitable versions of $(\th_t)_{t\in\N}$ and $(\th_t^\ast)_{t\in\N}$ are coupled to each other. Finally, the proof of the invariance principle (Theorem \ref{thm:CLT}) boils down to verifying conditions of Corollary 1 in Herrndorf's paper \cite{herrndorf1984}: we verify that the mixing coefficients decrease sufficiently fast and that the
covariance function of the process $\theta$ converges to its stationary counterpart.

\subsection{Drift and minorization conditions for $\theta_{t}$}

In this point, we establish suitable versions of the standard drift and minorization conditions, known from the theory of Markov chains (See e.g. \cite{mt}), for $(\theta_{t})_{t\in\N}$.
According to the next lemma, there is an $a=a(\lambda)>0$ such that for the Lyapunov function $V(\th) = \exp(a\lVert\th\rVert^2)$ and for the parametric kernel $Q$, a Foster--Lyapunov-type \emph{drift condition} holds.

\begin{lemma}\label{lem-qexp-drift}
	For any $0<\la<\frac{\Delta}{K^2}$, there exist $a>0$ such that for $V(\th)=\exp (a \lVert\th\rVert^2)$,
	$$	[Q(y)V](\th):=\int_{\X} V(z)\,Q(y,\th,\dint z) \le\ga V(\th)+C,\,\,\lVert y\rVert\le M $$
	holds with constants $\ga\in (0,1)$ and $C\ge 1$.
\end{lemma}
\begin{proof}
	We can write
	\begin{align*}
	[Q(y)V](\th) &= \E \left[\exp\left(a\lVert \th-\la H(\th,y)+\sqrt{2\la}\xi_0\rVert^2\right)\right] 
	=
		\frac{1}{(1-4\la a)^{d/2}}\exp\left({a\frac{\lVert \th-\la H(\th,y)\rVert^2}{1-4\la a}}\right),
	\end{align*}
	where by Assumption \ref{as:oldpaper}, for $\lVert y\rVert\le M$, we have
	$$
	\lVert \th-\la H(\th,y)\rVert^2 \le (2K^2\la^2-2\Delta\la+1)\lVert \th\rVert^2
	+
	2(\la b+\la^2 K^2 (1+M)^2).
	$$
	For $0<\la<\frac{\Delta}{K^2}$, $0<2K^2\la^2-2\Delta\la+1<1$ hence we can choose $a>0$ so small that
	$$
	\frac{2K^2\la^2-2\Delta\la+1}{1-4\la a}<1.
	$$
	
	To sum up, we obtained that there are $c_1,c_2>0$ such that $c_2<a$ and $[Q(y)V](\th)\le c_1 \exp (c_2 \lVert\th\rVert^2)$ hence for $r>0$ large enough
	$\ga:=c_1 e^{-(a-c_2)r^2}<1$, and thus
	$$
	[Q(y)V](\th)\le\ga V(\th)+C
	$$
	holds with $C=c_1 e^{c_2 r^{2}}$, which completes the proof.
\end{proof}

\begin{corollary}\label{cor:drift}
	By induction, easily follows that for any collection $\{y_a,y_{a+1},\ldots,y_b\}\subset \Y$, $\lVert y_i\rVert\le M$, $a\le i\le b$, we have
	\begin{equation}\label{eq:drift_it}
	[Q(y_b)\ldots Q(y_a)V](\th):=[Q(y_b)[\ldots [Q(y_a)V]\ldots]](\th)\le \ga^{b-a+1}V(\th)+\frac{C}{1-\ga}
	\end{equation}
	hence by the tower rule, we can estimate further and obtain 
	$$
	\sup_{t\in\N} \E (V (\th_t'))\le \E (V (\th_0'))+\frac{C}{1-\ga}<\infty
	$$
	for initial values $\th_0'$ satisfying $ \E (V (\th_0'))<\infty$.
\end{corollary}
From now on, let us fix a $\la\in (0,\Delta/K^2)$ and $a>0$ as in Lemma \ref{lem-qexp-drift}. 

\medskip
In the theory of Markov chains, the Foster-Lyapunov condition is often accompanied by 
a \emph{minorization condition} on suitable ``small sets''. In the current model we do have such a \emph{minorization condition} on 
every compact set. (In other words, compact sets are small.) 
To see this, for fixed $R>0$, $\lVert \th \rVert\le R$ and $\lVert y \rVert\le M$, we can write
\begin{align*}
Q(y,\th,A) &= \int_{\X} \ind_{\th-\la H(\th,y)+\sqrt{2\la}z\in A}f_{\xi_0}(z)\,\dint z =
\int_{\X} \ind_{u\in A}\frac{1}{(2\la)^{d/2}}f_{\xi_0}\left(
\frac{1}{\sqrt{2\la}}(\th-\la H(\th,y)-u) 
\right)\,\dint z \\
&\ge m_{R,M,\la,K} \times\mathrm{Leb} ( A\cap\{x\mid \lVert x\rVert\le R\}),
\end{align*}
where $f_{\xi_0}$ is the probability density function of $\xi_0$ and the positive constant $m_{R,M,\la,K}$ is given by
$$
m_{R,M,\la,K} = \inf \left\{
f_{\xi_0}(z)\middle| z\in\X,\,\lVert z\rVert\le \frac{(\la K+2)R+\la K (M+1)}{\sqrt{2\la}}
\right\}.
$$

We note this observation in the next lemma.

\begin{lemma}\label{lem:minor}
For every $R>0$, there is a Borel probability measure $\nu_R$ on $\B(\X)$ and a coefficient $\tilde{\alpha}_R\in (0,1)$ such that
for $\lVert y\rVert\le M$ and $\lVert \th \rVert\le R$,
\begin{equation}\label{eq:minor}
Q(y,\th,A) \ge \tilde{\alpha}_R \nu_R (A),\,\,A\in\B (\X).
\end{equation}
\end{lemma}

\subsection{Stationary initialization}

We need to show that, starting from a suitable random initial state $\th_0^\ast$, the process $(\theta_{t}^\ast)_{t\in\N}$ has a stationary version (in the strict sense).

Let $\M^Y$ be the set of Borel probability laws on $\X\times\Y^\Z$ such that their second marginal equals 
to the law of $(Y_t)_{t\in\Z}$, and $\M_b^Y$ denotes the set of those $\mu\in\M^Y$ for which the process 
$(\th'_t)_{t\in\N}$ started from some random initial state $\th'_0$ with $\law{(\th'_0,(Y_t)_{t\in\Z})}=\mu$ satisfies
\begin{equation}\label{eq:MbY}
	\sup_{t\in\N} \P (\lVert \th'_t\rVert\ge n)\to 0,\,n\to\infty.
\end{equation}

By Corollary \ref{cor:drift} and the Markov inequality, for every random variable $(\th'_0,(Y_t)_{t\in\Z})$ with law in $\M^Y$ and $\E (V(\th_0'))<\infty$, \eqref{eq:MbY} holds hence
$\law{(\th'_0,(Y_t)_{t\in\Z})}\in\M_b^Y$. In particular, for any deterministic $\th_0\in\X$, $\delta_{\th_0}\otimes\law{(Y_t)_{t\in\Z}}\in\M_b^Y$, where $\delta_{\th_0}$ stands 
for the Dirac measure concentrated on $\th_0$. It follows that $\M_b^Y\ne\emptyset$.

\begin{lemma}\label{lem:gbrm} 
For each $\law{(\th'_0,(Y_t)_{t\in\Z})}\in\M_b^Y$, there exists a limiting probability $\mu^\ast$ on $\B (\X\times\Y^\Z)$ such that
$$
\dtv (\law{(\th'_t,(Y_{k+t})_{k\in\Z})},\mu^\ast)\to 0,\,t\to\infty.
$$
In addition, $\mu^\ast$ does not depend on the choice of $(\th'_0,(Y_t)_{t\in\Z})$.
If $\law{(\theta^{*}_{0},(Y_{k})_{k\in\Z})}=\mu_{*}$ then the process $(\th_t^\ast,(Y_{k+t})_{k\in\Z})$, $t\in\N$ is (strict-sense) stationary.
\end{lemma}
\begin{proof}
The statement follows from R\'asonyi and Gerencs\'er's recent result, Theorem 3.10. in \cite{gerencser2022invariant}. They also prove that $\law{(\th_t^\ast,(Y_{k+t})_{k\in\Z})}=\mu^\ast$ for each $t\in\N$. Since $(\th^{*}_t,(Y_{t+k})_{k\in\Z})$, $t\in\N$ is a time-homogeneous Markovian process, strong stationarity follows. 
\end{proof}

\begin{remark}\label{rem:V0bd}
	{\rm By Corollary \ref{cor:drift} and Lemma \ref{lem:gbrm}, for any arbitrary but deterministic $\th_0\in\X$, we have 
	$$
	\E (V(\th_0^\ast)) = \lim_{\Sigma\to\infty} \E (\min (\Sigma,V(\th_0^\ast))) = 
	\lim_{\Sigma\to\infty} \lim_{t\to\infty} \E (\min (\Sigma,V(\th_t)))
	\le V(\th_0)+\frac{C}{1-\ga},
	$$
	and thus by the strong stationarity of $(Y_t)_{t\in\Z}$, immediately follows that $\mu^\ast\in\M_b^Y$.
	}
\end{remark}

\begin{remark}\label{rem:truquet}
{\rm With the above form of the drift and minorization condition in hand and using a recent result of Truquet's (Theorem 1 in \cite{Truquet1}), we could as well deduced the existence of a stationary process $(\th_t^\ast)_{t\in\Z}$ satisfying
$$
	\P (\th_{t+1}^\ast\in A\mid (\th_{i}^\ast)_{i\le t}, (Y_j)_{j\in\Z}) = Q(Y_t,\th_{t}^\ast,A),\,\,A\in\B (\X),\,\,t\in\Z.
$$
However, we will need a bit more. We aim to  show that there is a coupling between the iterations $(\th_t)_{t\in\N}$ initialized with the deterministic value $\th_{0}\in\X$ and an appropriate version of $(\th_t^\ast)_{t\in\Z}$. That's why we preferred the technology presented in \cite{gerencser2022invariant}.

It is also shown in \cite{Truquet1} that the distribution of $(\th_t^\ast)_{t\in\Z}$ is unique, moreover the process $(\th_t^\ast, Y_t)_{t\in\Z}$ is ergodic provided that $(Y_t)_{t\in\Z}$ is ergodic. The latter will be very important for us, since the proof of Theorem \ref{thm:CLT} relies on this result.	
	
In addition, Truquet proved that under a milder form of the drift and minorization conditions (See Assumptions A2 and A3 in \cite{Truquet1}), $\law{\th_t}\to\law{\th_0^\ast}$ in total variation as $t\to\infty$. However, as Truquet remarked, assumptions of \cite{Truquet1} did not to get a rate of convergence for $\law{\th_t}$.
}
\end{remark}

In the rest of this subsection, we show an alternative approach to a bit stronger result on the convergence of $(\law{\th_t})_{t\in\N}$ using the results of the recent paper \cite{balazs}. The reader can skip this part without affecting the understanding. The $(1+V)$-weighted total variation distance for any pair of Borel probability measures $\mu,\nu$ on $\B(\X)$ is defined by
$$
\dtv^{1+V}(\mu,\nu):= \int_{\X} (1+V(\th))|\mu-\nu|(\dint \th).
$$

\begin{lemma}\label{lem:weightedtv}
There exist constants $c_1,c_2>0$ such that for $V(\th) =e^{\frac{a}{2}\lVert\th\rVert^2}$,
\begin{equation*}
	\dtv^{1+V}(\law{\th_n},\law{\th_{0}^\ast})\le c_1 e^{-c_2 n},\,n\in\N.
\end{equation*}	
\end{lemma}
\begin{proof}
With the above choice of $V$, we have $\E (V(\th_{0})^2+V(\th_1)^2)<\infty$ hence the \emph{moment condition on initial values} i.e. Assumption 2.6. in \cite{balazs} is in force, and also the other assumptions of \cite{balazs} are clearly met (with the quantities $\lambda,\alpha,K$ constant
and with $\ell\equiv 0$ since $Y$ is bounded) too. Hence Theorem 2.11 of \cite{balazs} implies the convergence of $\law{\th_t}$ towards the limiting distribution $\law{\th_{0}^\ast}$ at a geometric rate in $\dtv^{1+V}$.  	
\end{proof}

\begin{corollary}\label{cor:moments}
It is clear from the definition of $d^{1+V}$ and from Lemma \ref{lem:weightedtv} that for any $\phi$ satisfying \eqref{eq:polygrowth},
$$
\E (\phi (\th_n))\to\E (\phi (\th_0^\ast)),\,\,n\to\infty.
$$
In particular, $\E (\lVert \th_n\rVert^p)\to\E (\lVert \th_0^\ast\rVert^p)$, as $n\to\infty$, for every $1\le p<\infty$.
\end{corollary}

\subsection{Coupling construction}

Let $R>0$ which we fix later, and $(\eps_t)_{t\in\Z}$ be a sequence of i.i.d. uniform variables on $[0,1]$ independent of $(Y_k)_{k\in\Z}$ and also independent of $(\xi_n)_{n\ge 1}$. The next Lemma is a standard representation results for parametric kernels satisfying the minorization condition 
\eqref{eq:minor}.
\begin{lemma}\label{lem:T}
Under the minorization condition (c.f. \eqref{eq:minor} in Lemma \ref{lem:minor}), there exists a measurable function $T:\Y\times\X\times [0,1]\to\X$ such that
$$
\P (T (y,\th,\eps_0)\in\A) = Q(y,\th,A),
$$
for all $\th\in\X$, $A\in\B(\X)$ and $y\in\Y$ such that $\lVert y\rVert\le M$. Furthermore, for $u\in [0,\tilde{\alpha}_R]$,
$$
T (y,\th_1,u) = T (y,\th_2,u),\,\,\lVert y\rVert\le M,\,\th_1,\th_2\in\{\th\mid \lVert\th\rVert\le R\}.
$$
\end{lemma}
\begin{proof}
	For the proof, we refer the reader to Lemma 7.1 in \cite{lovas}.
\end{proof}

We drop the dependence of the mappings T on $\eps_t$ in the notation and will simply
write $T_t(y)\th:= T(\th, y, \eps_t)$. For $s\in\Z$ and $\th\in\X$, define the family of auxiliary
processes
\begin{equation}\label{eq:Z}
	Z_{s,t}^{\th,\mathbf{y}} = \th,\,t\le s,
	\quad 
	Z_{s,t}^{\th,\mathbf{y}} = T_t(y_{t-1})Z_{s,{t-1}}^{\th,\mathbf{y}},\,t>s,
\end{equation}
where $\mathbf{y}=(\ldots,y_{-1},y_0,y_1,\ldots)\in\Y^\Z$ is a fixed trajectory. Clearly, for any random variable $(\th'_0,(Y_k)_{k\in\Z})$ and $s\in\N$, $Z_{s,t}^{\th'_s,\mathbf{Y}}$, $t\ge s$ is a version of the process $(\th'_t)_{t\in\N}$ defined through the iterative scheme \eqref{eq:theta}, starting from $\th'_0$ and driven by $(Y_k)_{k\in\Z}$.
Furthermore, the process $Z_{s,t}^{\th_0,\mathbf{y}}$, $t\ge s$ is a time-inhomogeneous Markov chain that follows the dynamics of $\theta_t$, $t\in\N$ with the environment being ''frozen''. Since the process $(Y_k)_{k\in\Z}$ is almost surely bounded by $M>0$, we can restrict ourselves to trajectories $\mathbf{y}\in\Y^\Z$ satisfying $\sup_{k\in\Z}\lVert y_k\rVert\le M$, and thus $Z_{s,t}^{\th_0,\mathbf{y}}$, $t\ge s$ is a \emph{Harris recurrent} chain. The next lemma controls the coupling time between processes starting from different initial values.
\begin{lemma}\label{lem:coupling}
	Let $\th_1,\th_2\in\X$ be arbitrary but fixed and $\mathbf{y}\in\Y^\Z$ such that $\sup_{k\in\Z}\lVert y_k\rVert\le M$. Then there exists constants $\ka>0$ and $N\in\N$
	depending only on $\la$, $\Delta$, $b$, $K$ and $M$ such that for $n\ge N$,
	$$
	\P (Z_{0,n}^{\th_1,\mathbf{y}}\ne Z_{0,n}^{\th_2,\mathbf{y}})\le
	\frac{V(\th_1)+V(\th_2)+3}{2}e^{-\ka n}.
	$$
\end{lemma}
\begin{proof}
	First, we fix $\ga<\ga'<1$ and choose $R>0$ so large such that $2C<(\ga'-\ga)e^{\frac{a}{2}R^2}$. Furthermore, we introduce the notations $\overline{Z}_n:=\left(Z_{0,n}^{\th_1,\mathbf{y}}, Z_{0,n}^{\th_2,\mathbf{y}}\right)$, $\lVert \overline{Z}_n\rVert:=\max \left( \left\lVert Z_{0,n}^{\th_1,\mathbf{y}} \right\rVert,
	\left\lVert Z_{0,n}^{\th_2,\mathbf{y}} \right\rVert\right)$ and the sequence of successive visiting times
	$$
	\sigma_0:=0,\,\sigma_{k+1} = \min\left\{t>\sigma_k\middle|  \lVert \overline{Z}_t\rVert\le R
	\right\},\,k\in\N
	$$
	that are obviously $\sigma (\eps_t,t\in\Z)$-stopping times. Note that on $\{ \lVert \overline{Z}_t\rVert>R\}$ we have 
	$$
	\ga (V( Z_{0,t}^{\th_1,\mathbf{y}})+V( Z_{0,t}^{\th_2,\mathbf{y}}))+2C 
	\le
	\ga' (V( Z_{0,t}^{\th_1,\mathbf{y}})+V( Z_{0,t}^{\th_2,\mathbf{y}}))
	$$ 
	and thus for $k\ge 1$ and $s\ge 0$, we obtain
	\begin{align*}
		\P (\sigma_{k+1}-\sigma_{k}>s\mid \overline{Z}_{\sigma_k})
		&\le 
		\E \left(
		\P (\lVert \overline{Z}_{\sigma_k+s}\rVert>R\mid \overline{Z}_{\sigma_{k}+s-1})
		\prod_{j=1}^{s-1}
		\ind_{\lVert \overline{Z}_{\sigma_k+j}\rVert>R}
		\middle| \overline{Z}_{\sigma_k}\right)\\
		&\le 
		\ga'
			\E \left(
		\frac{V( Z_{0,\sigma_k+s-1}^{\th_1,\mathbf{y}})+V( Z_{0,\sigma_k+s-1}^{\th_2,\mathbf{y}})}{e^{\frac{a}{2}R^2}}
		\prod_{j=1}^{s-2}
		\ind_{\lVert \overline{Z}_{\sigma_k+j}\rVert>R}
		\middle| \overline{Z}_{\sigma_k}\right)
	\end{align*}
	Iteration of this argument leads to the following estimation.
	\begin{align*}
	\P (\sigma_{k+1}-\sigma_{k}>s\mid \overline{Z}_{\sigma_k}) 
	&\le 
	(\ga')^{s-1}\frac{\ga V( Z_{0,\sigma_k}^{\th_1,\mathbf{y}})+\ga V( Z_{0,\sigma_k}^{\th_2,\mathbf{y}})+2C}{e^{\frac{a}{2}R^2}} \\
	&\le 
	(\ga')^{s-1} \frac{2\ga e^{\frac{a}{2}R^2}+2C}{e^{\frac{a}{2}R^2}}
	\le
	(\ga')^{s-1} (\ga'+\ga)\le 2(\ga')^s. 
	\end{align*}
	Along similar lines, we can show that 
	$$
	\P (\sigma_1>s)\le (\ga')^{s} \left[e^{-\frac{a}{2}R^2}(V(\th_1)+V(\th_2)) + 1-\frac{\ga}{\ga'}
	\right]\le (e^{-\frac{a}{2}R^2}(V(\th_1)+V(\th_2)) + 1)(\ga')^{s}.
	$$
	Let us fix $\ga''$ such that $\ga'<\ga''<1$. For the generating function of the time elapsed between the $k$th and $(k+1)$th visits, we get
	\begin{align*}
	\E \left(\frac{1}{(\ga'')^{\sigma_{k+1}-\sigma_k}}\middle| \F_{-\infty,\sigma_k}^\eps\right) = \sum_{j=1}^{\infty}
	\frac{1}{(\ga'')^{j}} \P (\sigma_{k+1}-\sigma_{k}=j\mid \overline{Z}_{\sigma_k}) 
	\le 
	\sum_{j=1}^{\infty} 	\frac{2 (\ga')^{j-1}}{(\ga'')^{j}} = \frac{2}{\ga''-\ga'},\,k\ge 1,
	\end{align*} 
	and similarly, for $k=0$,
	$$
	\E \left(\frac{1}{(\ga'')^{\sigma_{1}}}\right)\le \frac{e^{-\frac{a}{2}R^2}(V(\th_1)+V(\th_2)) + 1}{\ga''-\ga'}
	$$
	hence by the Markov inequality and the tower rule,for $0<m<n$,  we obtain
	\begin{align*}
		\P (\sigma_m\ge n)&\le (\ga'')^n \E \left(\frac{1}{(\ga'')^{\sigma_m}}\right) =
		(\ga'')^n \E \left( \E \left(\frac{1}{(\ga'')^{\sigma_m-\sigma_{m-1}}}
		\middle|\F_{-\infty,\sigma_{m-1}}^\eps 
		\right)  \frac{1}{(\ga'')^{\sigma_{m-1}}}\right)
		\\
		&\le
		(\ga'')^n\frac{2}{\ga''-\ga'}\E \left(\frac{1}{(\ga'')^{\sigma_{m-1}}}\right)
		\le\ldots\\
		&\le 
		\frac{e^{-\frac{a}{2}R^2}(V(\th_1)+V(\th_2)) + 1}{2}
		\left(\frac{2^m}{(\ga''-\ga')^m}\right)(\ga'')^n.
	\end{align*}
	Again we fix a constant $\ga'''$ such that $\ga''<\ga'''<1$, and define
	$$
	m_n := \left\lfloor n\frac{\log \ga'''-\log \ga''}{\log 2-\log (\ga''-\ga')}\right\rfloor.
	$$
	Obviously, for $n$ is so large such that $m_n\ge 1$, we have
	$$
	\P (\sigma_{m_n}\ge n)\le \frac{e^{-\frac{a}{2}R^2}(V(\th_1)+V(\th_2)) + 1}{2} (\ga''')^n.
	$$

	Next, we estimate the probability of no-coupling on events when the small set is visited at least $m_n$-times. According to Lemma \ref{lem:T}, for $j=1,\ldots,m_n$, $\th\mapsto T(y,\th,\eps_{\sigma_j+1})$ is constant of the ball $\{\th\mid \lVert\th\rVert\le R\}$ with probability at least $\tilde{\alpha}_R$ hence we can write
	\begin{align*}
		\P (Z_{0,n}^{\th_1,\mathbf{y}}\ne Z_{0,n}^{\th_2,\mathbf{y}},\sigma_{m_n}<n) 
		\le
		\P (\eps_{\sigma_j+1}<\tilde{\alpha}_R;\, j= 1,\ldots, m_n) = \tilde{\alpha}_R^{m_n}, 
	\end{align*}
	where we used that for every $j$, $\eps_{\sigma_j+1}$ is independent of $\F_{-\infty,\sigma_j}^\eps$.
	
	At least, we combine this estimate with that one what we got for the tail probability of the visiting times, and obtain
	\begin{align*}
		\P (Z_{0,n}^{\th_1,\mathbf{y}}\ne Z_{0,n}^{\th_2,\mathbf{y}})
		&\le 
		\P (Z_{0,n}^{\th_1,\mathbf{y}}\ne Z_{0,n}^{\th_2,\mathbf{y}},\sigma_{m_n}<n)
		+
		\P (\sigma_{m_n}\ge n)\\
		& \le \tilde{\alpha}_R^{m_n} + \frac{e^{-\frac{a}{2}R^2}(V(\th_1)+V(\th_2)) + 1}{2} (\ga''')^n
	\end{align*}
	which completes the proof.
	
\end{proof}

The following annealed version of Lemma \ref{lem:coupling} will be important later.
\begin{lemma}\label{lem:annealed}
Let $\theta_{1},\theta_{2}$ be random variables independent of $\mathcal{F}^{\eps}_{m+1,\infty}$ for some $m\in\N$. Then
$$
\P (Z_{m,n}^{\th_1,\mathbf{Y}}\ne Z_{m,n}^{\th_2,\mathbf{Y}})\le
\frac{\E(V(\th_1))+\E(V(\th_2))+3}{2}e^{-\ka (n-m)},\,\,n\ge N+m.
$$
\end{lemma}
\begin{proof}
Estimate the conditional probability using Lemma \ref{lem:coupling} above yields
\begin{align*}
\P (Z_{m,n}^{\th_1,\mathbf{Y}}\ne Z_{m,n}^{\th_2,\mathbf{Y}}\mid \mathbf{Y}=\mathbf{y},\,\th_1=x_1,\,\th_2=x_2) 
&=
\P (Z_{m,n}^{x_1,\mathbf{y}}\ne Z_{m,n}^{x_2,\mathbf{y}})
=
\P \left(
Z_{0,n-m}^{x_1,S^m\mathbf{y}}\ne Z_{0,n-m}^{x_2,S^m\mathbf{y}}
\right)
\\
&\le
\frac{V(x_1)+V(x_2)+3}{2}e^{-\ka (n-m)},\,\, n\ge N+m,
\end{align*}
where $S^{m}\mathbf{Y}$ refers to the $m$-times left-shifted trajectory of $Y$ i.e. $\left(S^m\mathbf{Y}\right)_k=Y_{k+m}$, $k\in\Z$. Finally, we take expectations and obtain the claimed inequality.
\end{proof}

\subsection{Mixing properties}

In what follows, we show that mixing properties of the exogenous environment transfer to the process $\th_t$, $t\in\N$. For any system of sub-$\sigma$-algebras $\A_i\subset\F$, $i\in I$, we use the notation $\bigvee_{i\in I}A_i$ for the $\sigma$-algebra generated by the system $(\A_i)_{i\in I}$. 
\begin{lemma}\label{lem:salg}
	Suppose $\A_n$ and $\B_n$, $n=1,2,\ldots$ are sub-$\sigma$-algebras of $\F$ such that the $\sigma$-algebras $\A_n\vee\B_n$, $n=1,2,\ldots$ are pairwise independent. Then
	$$
	\alpha \left(\bigvee_{n=1}^\infty\A_n,\bigvee_{n=1}^\infty\B_n\right)
	\le\sum_{n=1}^\infty \alpha (\A_n, \B_n).
	$$
\end{lemma}

\begin{proof}
	The proof can be found in \cite[Lemma 8 on page 13]{BRADLEY19811}.
\end{proof}

\begin{remark}\label{rem:mix}
	We need the special case when $\A_1,\A_2,\B_1,\B_2\subset\F$ are $\sigma$-algebras, where $\A_2$ and $\B_2$ are independent too. For this, Lemma \ref{lem:salg} gives
	$\alpha (\A_1\vee\A_2,\B_1\vee\B_2)\le \alpha (\A_1,\B_1)$. By the definition of the measure of dependence between sigma algebras \eqref{eq:dep}, the reverse inequality trivially holds hence
	$$
	\alpha (\A_1\vee\A_2,\B_1\vee\B_2)= \alpha (\A_1,\B_1).
	$$
\end{remark}

The next lemma provides an upper bound for the strong mixing coefficient of the chain $(\theta_t)_{t\in\N}$ given $\alpha^Y$.

\begin{lemma}\label{lem:mixing}
	For the dependence coefficient $\alpha_j^\th (n)$, we have the following upper estimate
	$$
		\alpha_j^\th (n) \le \alpha^Y (\lfloor n/2\rfloor)
	+ 
	\left(V(\th_{0})+\frac{3}{2}+\frac{C}{2(1-\ga)}\right)e^{-\frac{\ka}{2}n},\,\,j\ge 0,\,n\ge 2N,
	$$
	where $\kappa$ and $N$ are as in Lemma \ref{lem:coupling}.
\end{lemma}
\begin{proof}
	We introduce the notations $\th_{\rightarrow t}:=(\th_0,\th_1,\ldots,\th_t)$ and $\th_{t\rightarrow}:=(\th_t,\th_{t+1},\ldots)$, and also $Z_{s,\rightarrow t}^{\th_0,\mathbf{y}}:=(Z_{s,s}^{\th_0,\mathbf{y}},\ldots,Z_{s,t}^{\th_0,\mathbf{y}})$,
	$Z_{s,t\rightarrow}^{\th_0,\mathbf{y}}:=(Z_{s,t}^{\th_0,\mathbf{y}},Z_{s,t+1}^{\th_0,\mathbf{y}},\ldots)$. Let $A\in\F_{0,j}^\th$ and $B\in\F_{j+n,\infty}^\th$ be arbitrary events. Then by the definition of the generated $\sigma$-algebra, exist $A^\X\in\B(\X^{j+1})$ and $B^\X\in\B(\X^\N)$ such that
	$$
	A=\left\{\omega\in\Omega\middle| \th_{\rightarrow j}(\omega)\in A^\X  \right\}\,\,\text{and}\,\,B=\left\{\omega\in\Omega\middle|\th_{j+n\rightarrow}(\omega)\in B^\X  \right\}.
	$$
	So, for any $r_n$ satisfying $0\le r_n\le n-N$ we can write
	\begin{align}\label{eq:alpha}
	\begin{split}
		|\P (A\cap B)-\P(A)\P(B)| &= |\cov (\ind_{\th_{\rightarrow j}\in A^\X},\ind_{\th_{j+n\rightarrow}\in B^\X})|
		=
		\left|
		\cov (\ind_{Z_{0,\rightarrow j}^{\th_0,\mathbf{Y}}\in A^\X},\ind_{Z_{0, j+n\rightarrow}^{\th_0,\mathbf{Y}}\in B^\X})
		\right| 
		\\
		&\le 
		\left|
		\cov (\ind_{Z_{0,\rightarrow j}^{\th_0,\mathbf{Y}}\in A^\X},\ind_{Z_{j+r_n, j+n\rightarrow}^{\th_0,\mathbf{Y}}\in B^\X})
		\right|
		+
		\P \left(Z_{0, j+n}^{\th_0,\mathbf{Y}}\ne Z_{j+r_n, j+n}^{\th_0,\mathbf{Y}}\right).
	\end{split}
	\end{align}
	
	Observe that $\ind_{Z_{0,\rightarrow j}^{\th_0,\mathbf{Y}}\in A^\X}$ is $\F_{-\infty,j-1}^Y\vee\F_{1,j}^\eps$-measurable and $\ind_{Z_{j+r_n, j+n\rightarrow}^{\th_0,\mathbf{Y}}\in B^\X}$ is $\F_{j+r_n,j+n-1}^Y\vee\F_{j+r_n+1,j+n}^\eps$-measurable, moreover
	$\F_{-\infty,j-1}^Y\vee\F_{j+r_n,j+n-1}^Y$ is independent of $\F_{1,j}^\eps\vee\F_{j+r_n+1,j+n}^\eps$, and also the $\sigma$-algebras $\F_{1,j}^\eps$ and $\F_{j+r_n+1,j+n}^\eps$ are independent of each other hence by Remark \ref{rem:mix} and the stationarity of $(Y_k)_{k\in\Z}$, we have
	\begin{align}\label{eq:aY}
	\begin{split}
	\left|
	\cov (\ind_{Z_{0,\rightarrow j}^{\th_0,\mathbf{Y}}\in A^\X},\ind_{Z_{j+r_n, j+n\rightarrow}^{\th_0,\mathbf{Y}}\in B^\X})
	\right|
	&\le 
	\alpha
	\left(
	\F_{-\infty,j-1}^Y\vee\F_{1,j}^\eps,
	\F_{j+r_n,j+n-1}^Y\vee\F_{j+r_n+1,j+n}^\eps
	\right) \\
	&\le 
	\alpha
	\left(
	\F_{-\infty,j-1}^Y,
	\F_{j+r_n,j+n-1}^Y
	\right) \le \alpha_{j-1}^Y (r_n+1) \\
	&=\alpha^Y (r_n+1).
	\end{split}
	\end{align}
	By Lemma \ref{lem:annealed} and Corollary \ref{cor:drift}, we can estimate the second term on the right-hand side of \eqref{eq:alpha}
	\begin{align*}
		\P \left(Z_{0, j+n}^{\th_0,\mathbf{Y}}\ne Z_{j+r_n, j+n}^{\th_0,\mathbf{Y}}\right)
		&=
		\P\left(Z_{j+r_n, j+n}^{\th_{j+r_n},\mathbf{Y}}\ne Z_{j+r_n, j+n}^{\th_0,\mathbf{Y}}\right)
		\le 
		\frac{\E (V(\th_{j+r_n}))+V(\th_{0})+3}{2}e^{-\ka (n-r_n)} \\
		&\le 
		\left(V(\th_{0})+\frac{3}{2}+\frac{C}{2(1-\ga)}\right)e^{-\ka (n-r_n)}.
	\end{align*}
	Combining this with \eqref{eq:aY}, and taking the supremum on the left-hand side of \eqref{eq:alpha} yields
	\begin{equation*}
		\alpha_j^\th (n)\le \alpha^Y (r_n+1) + \left(V(\th_{0})+\frac{3}{2}+\frac{C}{2(1-\ga)}\right)e^{-\ka (n-r_n)}.
	\end{equation*}
	for any $0\le r_n\le n-N$. By choosing $r_n = \lfloor n/2\rfloor$, we obtain the desired inequality
	$$
	\alpha_j^\th (n) \le \alpha^Y (\lfloor n/2\rfloor)
	+ 
	\left(V(\th_{0})+\frac{3}{2}+\frac{C}{2(1-\ga)}\right)e^{-\frac{\ka}{2}n},\,\,j\ge 0,\,n\ge 2N.
	$$
\end{proof}

\subsection{Proof of Theorem \ref{thm:LLN}}

\begin{lemma}\label{lem:UI}
The sequence $(\phi (\th_t))_{t\in\N}$ is uniformly $L^p$-bounded for every $1\le p<\infty$ that is
	\begin{equation}\label{eq:HD1}
	c_p:=\sup_{t\in\N}\E^{1/p}(|\phi (\th_t)|^p)<\infty.	
	\end{equation}
\end{lemma}
\begin{proof}
 	Using $x^s\le \Gamma (s+1) e^x$, $x,s\ge 0$ and \eqref{eq:polygrowth},  by Corollary \ref{cor:drift}, we can write
 	\begin{align*}
 		\E^{1/p}(|\phi (\th_t)|^p)
 		&\le 
 		c_{\phi}\left(1+\Gamma \left(\frac{rp}{2}+1\right)^{1/p}\E^{1/p} (V(\th_t))\right)
 		\le
 		c_{\phi}\left(1+\Gamma \left(\frac{rp}{2}+1\right)^{1/p}\left(V(\th_0)+\frac{C}{1-\ga}\right)^{1/p}\right),
 	\end{align*}
 	where the upper bound does not depend on $t$ hence \eqref{eq:HD1} clearly holds.
\end{proof}

\begin{proof}[Proof of Theorem \ref{thm:LLN}]
	Let $(\th_0^\ast, (Y_k)_{k\in\Z})$ be a random variable such that $\law{\left((\th_0^\ast, (Y_k)_{k\in\Z})\right)}=\mu^\ast$. For $k\in\N$, $0\le i_1<\ldots<i_k$ and $A\in\B (\X^k)$ arbitrary, we can write
	\begin{align*}
		\P \left(
		\left(\th_{i_1+n},\ldots,\th_{i_k+n}\right)\in A
		\right) 
		&=
		\P \left(
		\left(
		Z_{0,i_1+n}^{\th_0,\mathbf{Y}},\ldots,Z_{0,i_k+n}^{\th_0,\mathbf{Y}}\right)
		\in A
		\right)
		\\
		&\le 
		\P \left(
		\left(
		Z_{0,i_1+n}^{\th_0^\ast,\mathbf{Y}},\ldots,Z_{0,i_k+n}^{\th_0^\ast,\mathbf{Y}}
		\right)
		\in A
		\right) + 
		\P \left(Z_{0,i_1+n}^{\th_0,\mathbf{Y}}\ne Z_{0,i_1+n}^{\th_0^\ast,\mathbf{Y}} \right)
		\\
		&\le 
		\P \left(
		\left(\th_{i_1+n}^\ast,\ldots,\th_{i_k+n}^\ast\right)\in A
		\right) + \P \left(Z_{0,i_1+n}^{\th_0,\mathbf{Y}}\ne Z_{0,i_1+n}^{\th_0^\ast,\mathbf{Y}} \right).
	\end{align*}	
	By interchanging the role of $(\th_t)_{t\in\N}$ and  $(\th_t^\ast)_{t\in\N}$, we obtain
	$$
	\left| \P \left(
	\left(\th_{i_1+n},\ldots,\th_{i_k+n}\right)\in A
	\right)
	-
	\P \left(
	\left(\th_{i_1+n}^\ast,\ldots,\th_{i_k+n}^\ast\right)\in A
	\right)
	\right| 
	\le 
	\P \left(Z_{0,i_1+n}^{\th_0,\mathbf{Y}}\ne Z_{0,i_1+n}^{\th_0^\ast,\mathbf{Y}} \right),\,\,A\in\B(\X^k).
	$$
	Next, we take supremum on the left hand-side in $A\in\B(\X^k)$ and then by Lemma \ref{lem:annealed} and Remark \ref{rem:V0bd}, we arrive at
	\begin{align*}
	\dtv\left(\law{(\th_{i_1+n},\ldots,\th_{i_k+n})},\law{(\th_{i_1}^\ast,\ldots,\th_{i_k}^\ast)}\right)
	&\le 
	\P \left(Z_{0,i_1+n}^{\th_0,\mathbf{Y}}\ne Z_{0,i_1+n}^{\th_0^\ast,\mathbf{Y}} \right)
	\\
	&\le 
	\frac{V(\th_0)+\E (V(\th_0^\ast))+3}{2}e^{-\ka (i_1+n)}
	\\
	&\le 
	\left(V(\th_0)+\frac{3}{2}+\frac{C}{2(1-\ga)}\right) e^{-\ka n},\,\,n\ge N,
	\end{align*}
	where $N\in\N$ is as in Lemma \ref{lem:coupling} and \ref{lem:annealed}.
	
	\medskip
	In what follows, we proceed with the proof of the law of large numbers both in strong and $L^p$ sense. Again by Lemma \ref{lem:annealed} and Remark \ref{rem:V0bd}, exist an almost surely finite random variable $\tau$ such that
	$$
	Z_{0,n}^{\th_0,\mathbf{Y}}= Z_{0,n}^{\th_0^\ast,\mathbf{Y}},\,n\ge\tau.
	$$
	Furthermore, for the tail distribution of $\tau$,
	$$
	\P(\tau\ge n)\le\P \left(Z_{0,n}^{\th_0,\mathbf{Y}}\ne Z_{0,n}^{\th_0^\ast,\mathbf{Y}}\right) \le \left(V(\th_0)+\frac{3}{2}+\frac{C}{2(1-\ga)}\right) e^{-\ka n},\,\,n\ge N.
	$$
	holds with constants $\ka>0$ and $N\in\N$ as in Lemma \ref{lem:coupling} and \ref{lem:annealed}. 
	
	When the data stream $(Y_k)_{k\in\Z}$ is ergodic, by Remark \ref{rem:truquet}, the process $Z_{0,n}^{\th_0^\ast,\mathbf{Y}}$, $n\in\N$ is also ergodic, moreover by Remark \ref{rem:V0bd}, $\E (\phi (\th_0^\ast))<\infty$ for any $\phi:\X\to\R$ satisfying \eqref{eq:polygrowth} hence by Birkhoff's ergodic theorem,
	$$
	\frac{\phi (Z_{0,0}^{\th_0^\ast,\mathbf{Y}})+\ldots+\phi (Z_{0,n-1}^{\th_0^\ast,\mathbf{Y}})}{n} \to \E (\phi (\th_0^\ast)),\,n\to\infty,\,\P-\Pas
	$$
	Combining this with the above result on the almost surely finite coupling time yields the strong law of large numbers for $\phi\left(Z_{0,n}^{\th_0,\mathbf{Y}}\right)$, $n\in\N$. As we mentioned earlier, the discrete-time processes $(\th_n)_{n\in\N}$ and $\left(Z_{0,n}^{\th_0,\mathbf{Y}}\right)_{n\in\N}$ are versions of each other hence the strong law of large numbers holds for $(\phi (\th_n))_{n\in\N}$, as well.
	
	Finally, by Lemma \ref{lem:UI}, the sequence $\frac{1}{n}(\phi (\th_0)+\ldots+\phi (\th_{n-1}))$, $n\ge 1$ is uniformly integrable on every power $1\le p<\infty$, and thus the law of large numbers holds in $L^p$-sense too for $1\le p<\infty$ which completes the proof.
\end{proof}

\subsection{Proof of Theorem \ref{thm:CLT}}

The subsequent lemma establishes a stability result for the autocovariance function of the sequence $(\phi(\th_k))_{k\in\N}$. Additionally, it provides an explicit upper bound for $\sup_{k\in\N}|\cov (\phi (\th_k),\phi (\th_{k+l}))|$ in terms of
the $\alpha$-mixing coefficient of $Y$. 
For the sake of readability, the proof is delegated to Appendix \ref{ap}.

\begin{lemma}\label{lem:autocov}
	The autocovariance function of $(\phi(\th_k))_{k\in\N}$ has the following properties.
	\begin{enumerate}[i)]
		\item For every $l\in\N$, $\cov (\phi (\th_k),\phi (\th_{k+l}))\to \cov (\phi (\th_0^\ast),\phi (\th_l^\ast))$, as $k\to\infty$.
		
		\item There exists a constant $\Lambda>0$ depending only on $\Delta$, $b$, $K$, $M$ and $\la$ such that for any $k,l\in\N$,
		$$
		|\cov (\phi (\th_k),\phi (\th_{k+l}))|\le \Lambda \left(
		\alpha^Y (\lfloor l/2\rfloor)^{1-\epsilon}+e^{-\frac{\ka}{4} \lfloor l/2\rfloor}
		\right),
		$$
		where $\epsilon>0$ is as in Assumption \ref{as:mixing}.
	\end{enumerate}
\end{lemma}

\begin{proof}[Proof of Theorem \ref{thm:CLT}]
We are going to verify the conditions of Corollary 1 in Herrndorf's paper \cite{herrndorf1984}. By Lemma \ref{lem:UI}, for $X_n:=\phi (\th_n)-\E (\phi (\th_n))$, $n\in\N$, and for any $1\le p<\infty$,
$$
\sup_{n\in\N}\E^{1/p}(|X_n|^p)<\infty.
$$

Next, we prove for $S_n:=X_1+\ldots+X_n$, $\lim_{n\to\infty} \E S_n^2/n = \sigma^2$
holds with some $\sigma\ge 0$. For this, we consider the decomposition 
\begin{equation}\label{eq:dcp}
\frac{1}{n}\E S_n^2= \frac{1}{n}\sum_{k=1}^{n}\E (X_k^2) + \frac{2}{n}\sum_{1\le k<l\le n} \E (X_k X_l),
\end{equation}
where by point i) in Lemma \ref{lem:autocov}, $\E (X_k^2)\to\D^2 (\phi (\th_0^\ast))$ as $k\to\infty$ hence the first term on the right-hand side of \eqref{eq:dcp} converges to $\D^2 (\phi (\th_0^\ast))$. Regarding the second term, we introduce $A_{n,l}:=\frac{1}{n}\sum_{k=1}^{n-l} \E (X_k X_{k+l})$, $1\le l<n$, and define 
$$
b_n:=\frac{1}{n}\sum_{1\le k<l\le n} \E (X_k X_l)=
\sum_{l=1}^{n-1}\frac{1}{n}\sum_{k=1}^{n-l}\E (X_k X_{k+l})=
\sum_{l=1}^{n-1} A_{n,l}. 
$$
By point ii) in Lemma \ref{lem:autocov}, we have
\begin{equation}\label{eq:dcp2}
|A_{n,l}|\le \frac{1}{n}\sum_{k=1}^{n-l} |\E (X_k X_{k+l})|\le \Lambda \left(
\alpha^Y (\lfloor l/2\rfloor)^{1-\epsilon}+e^{-\frac{\ka}{4} \lfloor l/2\rfloor}
\right)
\end{equation}
hence due to Assumption \ref{as:mixing}, for any $\delta>0$, exists $\tilde{N}_\delta\in\N$ such that $\sum_{l=\tilde{N}_{\delta}}^{n-1} |A_{n,l}|<\delta$, $n>\tilde{N}_{\delta}$, and thus for $m,n>\tilde{N}_{\delta}$, we have
$$
|b_n-b_m|\le \sum_{l=1}^{\tilde{N}_{\delta}-1}|A_{n,l}-A_{m,l}| 
+
\sum_{l=\tilde{N}_{\delta}}^{n-1} |A_{n,l}|
+
\sum_{l=\tilde{N}_{\delta}}^{m-1} |A_{m,l}| 
<
\sum_{l=1}^{\tilde{N}_{\delta}-1}|A_{n,l}-A_{m,l}| 
+
2\delta.
$$
By point i) in Lemma \ref{lem:autocov}, for every $1\le l<\tilde{N}_{\delta}$, $A_{n,l}\to \cov (\phi (\th_0^\ast),\phi (\th_l^\ast))$, as $n\to\infty$, and since $\delta>0$ was arbitrary, we obtain that $(b_n)_{n\ge 1}$ is a Cauchy sequence.

At last, by Lemma \ref{lem:mixing}, for the mixing coefficient $\alpha_j^X(n)$, we have
$$
\alpha_j^X(n) = \alpha_j^{\phi\circ\theta}(n)\le \alpha_j^\theta (n)
\le 
\alpha^Y (\lfloor n/2\rfloor)
+ 
\left(V(\th_{0})+\frac{3}{2}+\frac{C}{2(1-\ga)}\right)e^{-\frac{\ka}{2}n},\,\,j\ge 0,\,n\ge 2N,
$$
and thus for $\alpha^X(n):=\sup_{j\in\N}\alpha_j^X(n)$, $n\in\N$, $\sum_{n=0}^\infty \alpha^X(n)^{1-\epsilon}<\infty$ holds.

To sum up, we have shown that all the conditions of Corollary 1 in \cite{herrndorf1984} are satisfied hence we can conclude that,
if $\sigma>0$ then
the sequence of random function $(B_n)_{n\ge 1}$ given by
$$
B_n(t) = \frac{S_{\lfloor nt\rfloor}}{\sigma\sqrt{n}},\,\,t\in [0,1],\,n\ge 1
$$
is weakly convergent to a standard Brownian motion $B$ on $D[0,1]$ endowed with the Skorohod topology which completes the proof.
If $\sigma=0$ then $\E(S_{n}^{2})/n\to 0$ implies ${S_{\lfloor nt\rfloor}}/{\sqrt{n}}\to 0$ in probability, for all $t\in [0,1]$,
hence also in $D[0,1]$.

\end{proof}

\appendix

\section{Proof of Lemma \ref{lem:autocov}}\label{ap}

The following auxiliary result is a variation of Theorem 17.2.2 from the renowned book by Ibragimov and Linnik \cite{IbragimovLinnik}. In the interest of self-contained explanation and for future reference, we have chosen to present this result here in the required form.
\begin{lemma}\label{lem:ibragimov}
	Let $\Psi_1,\Psi_2:\Omega\to\R$ be random variables such that for $\sigma$-algebras $\A_1$ and $\A_2$, $\Psi_i$ is $\A_i$-measurable, $i=1,2$. Furthermore, for some $0<\epsilon<1$,
	$\E (|\Psi_i|^{2/\epsilon+1})<c,\,\,i=1,2$ holds with some $c>0$. Then
	$$
	|\cov (\Psi_1,\Psi_2)|\le (4+5c)\alpha (\A_1, \A_2)^{1-\epsilon}.
	$$
\end{lemma}

\begin{proof}

Let $L\ge 1$ be a positive number which we will fix later, and define the truncated random variables
$$
\hat{\Psi}_i = \Psi_i \ind_{|\Psi_i|\le L}
\,\,\text{and}\,\,
\tilde{\Psi}_i = \Psi_i \ind_{|\Psi_i|> L},
\,\,i=1,2.
$$
We can estimate
\begin{align*}
|\cov (\hat{\Psi}_1,\hat{\Psi}_2)| 
&= 
\left|
\E \left[
\hat{\Psi}_1
\left(
\E (\hat{\Psi}_2\mid \A_1)
-
\E (\hat{\Psi}_2)
\right)
\right]
\right|
\le 
L \E \left(\left|\E (\hat{\Psi}_2\mid \A_1)
-
\E (\hat{\Psi}_2)\right|\right) \\
&=
L \E \left[\zeta_1\left(\E (\hat{\Psi}_2\mid \A_1)
-
\E (\hat{\Psi}_2)\right)\right] = L\cov (\zeta_1, \hat{\Psi}_2),
\end{align*}
where $\zeta_1=\mathrm{sgn}\left(\E (\hat{\Psi}_2\mid \A_1)
-
\E (\hat{\Psi}_2)\right)$ is $\A_1$-measurable hence by interchanging the role of $\zeta_1$ and $\hat{\Psi}_2$, we can apply the same argument, and thus obtain $|\cov (\hat{\Psi}_1,\hat{\Psi}_2)| \le L^2 |\cov (\zeta_1,\zeta_2)|$. Note that $\zeta_1$ and $\zeta_2$ take values in $\{-1,+1\}$. So, let $A_1 = \{\zeta_1=+1\}$, $A_2 = \{\zeta_1=-1\}$, $B_1 = \{\zeta_2=+1\}$, and $B_2 = \{\zeta_2=-1\}$, where $A_i$-s are $\A_1$-measurable, and $B_i$-s are $\A_2$-measurable. We can write
\begin{align*}
\cov (\zeta_1,\zeta_2) &= 
\P (A_1\cap B_1) + \P (A_2\cap B_2)
- \P (A_1\cap B_2) - \P (A_2\cap B_1)-(\P(A_1)-\P(A_2))(\P(B_1)-\P(B_2)) \\
&\le \sum_{i,j=1}^2|\P (A_i\cap B_j)-\P(A_i)\P(B_j)|\le 4\alpha (\A_1,\A_2),
\end{align*}
and thus we arrive at $|\cov (\hat{\Psi}_1,\hat{\Psi}_2)|\le 4L^2 \alpha (\A_1,\A_2)$.
Trivially, $\E (|\tilde{\Psi}_i|)\le cL^{-2/\epsilon}$ and $\E (\tilde{\Psi}_i^2)\le c L^{1-2/\epsilon}$, $i=1,2$ hence for $i,j=1,2$, $|\cov (\hat{\Psi}_i,\tilde{\Psi}_j)|\le 2L\E(|\tilde{\Psi}_j|)\le 2cL^{1-2/\epsilon}$, and also by the Cauchy--Schwartz inequality, we get
\begin{align*}
|\cov (\Psi_1,\Psi_2)|
&\le
|\cov (\hat{\Psi}_1,\hat{\Psi}_2)|
+
|\cov (\hat{\Psi}_1,\tilde{\Psi}_2)|
+
|\cov (\tilde{\Psi}_1,\hat{\Psi}_2)|
+
|\cov (\tilde{\Psi}_1,\tilde{\Psi}_2)| 
\\
&\le
4L^2 \alpha (\A_1,\A_2) +4cL^{1-2/\epsilon}
+ \E (\tilde{\Psi}_1^2)^{1/2} \E (\tilde{\Psi}_2^2)^{1/2}
\le 4L^2 \alpha (\A_1,\A_2) + 5cL^{1-2/\epsilon}.
\end{align*}
At last, we set $L=\alpha (\A_1,\A_2)^{-\epsilon/2}$, and since $\alpha (\A_1,\A_2)\le 1$, we obtain
\begin{align*}
|\cov (\Psi_1,\Psi_2)| &\le 4 \alpha (\A_1,\A_2)^{1-\epsilon} + 5c\alpha (\A_1,\A_2)^{1-\epsilon/2}\le (4+5c)\alpha (\A_1,\A_2)^{1-\epsilon},
\end{align*}
which completes the proof.
	
\end{proof}

\begin{proof}[Proof of Lemma \ref{lem:autocov}]
	
i) Let $l\in\N$ be arbitrary. Then by the Cauchy--Schwartz inequality and Lemma \ref{lem:UI}, we can write
\begin{align*}
|\E (\phi (\th_k)\phi (\th_{k+l}))-\E (\phi (\th_k^\ast)\phi (\th_{k+l}^\ast))|
\le c_2 \left[\E^{1/2}((\phi (\th_k)-\phi (\th_k^\ast))^2) + 
\E^{1/2}((\phi (\th_{k+l})-\phi (\th_{k+l}^\ast))^2)\right].
\end{align*}
Let $k\ge N$, where $N$ is as in Lemma \ref{lem:coupling} and \ref{lem:annealed}. We can estimate further by	
\begin{align*}
\E^{1/2}((\phi (\th_k)-\phi (\th_k^\ast))^2) 
&= \E^{1/2}
\left[
\left(
\phi \left(Z_{0,k}^{\th_0,\mathbf{y}}\right)
-
\phi \left(Z_{0,k}^{\th_0^\ast,\mathbf{y}}\right)
\right)^2
\ind_{Z_{0,k}^{\th_0,\mathbf{y}}\ne Z_{0,k}^{\th_0^\ast,\mathbf{y}}}
\right]
\\
&\le (\E^{1/4}(\phi (\th_k)^4)+\E^{1/4}(\phi (\th_k^\ast)^4)) \P\left(Z_{0,k}^{\th_0,\mathbf{Y}}\ne Z_{0,k}^{\th_0^\ast,\mathbf{Y}}\right)^{1/4}
\\
&\le
2c_4 \left(V(\th_0)+\frac{3}{2}+\frac{C}{2(1-\ga)}\right)^{1/4} e^{-\frac{\ka}{4} k}
\end{align*}	
Along similar lines, one can show that the same upper bound works for $\E^{1/2}((\phi (\th_{k+l})-\phi (\th_{k+l}^\ast))^2)$ as well, and thus we get
$$
|\E (\phi (\th_k)\phi (\th_{k+l}))-\E (\phi (\th_k^\ast)\phi (\th_{k+l}^\ast))|
\le 
4 c_2 c_4 \left(V(\th_0)+\frac{3}{2}+\frac{C}{2(1-\ga)}\right)^{1/4} e^{-\frac{\ka}{4} k}.
$$
For $l=0$ and $\sqrt{\phi(\theta_{t})}$, we obtain $\E (\phi (\th_k))\to\E (\phi (\th_0^\ast))$, as $k\to\infty$ hence we can conclude that $\cov (\phi (\th_k),\phi (\th_{k+l}))\to \cov (\phi (\th_0^\ast),\phi (\th_l^\ast))$, as $k\to\infty$.

\bigskip
\noindent
ii) We estimate
\begin{equation}\label{eq:terms}
|\cov (\phi (\th_k),\phi (\th_{k-l}))| 
\le
\left|\cov \left(
\phi(Z_{k-\lfloor l/2\rfloor,k}^{\th_0,\mathbf{Y}}),
\phi(Z_{0,k-l}^{\th_0,\mathbf{Y}})\right)\right|+
\left|\cov \left(
\phi(Z_{0,k}^{\th_0,\mathbf{Y}})-\phi(Z_{k-\lfloor l/2\rfloor,k}^{\th_0,\mathbf{Y}}),
\phi(Z_{0,k-l}^{\th_0,\mathbf{Y}})\right)\right|.
\end{equation}

Regarding the first term on the right-hand side of \eqref{eq:terms},  $\Psi_1:=\phi(Z_{0,k-l}^{\th_0,\mathbf{Y}})$ is $\F_{-\infty, k-l-1}^Y\vee \F_{-\infty,k-l}^\eps$-measurable and $\Psi_2:=\phi(Z_{k-\lfloor l/2\rfloor,k}^{\th_0,\mathbf{Y}})$ is $\F_{k-\lfloor l/2\rfloor,\infty}^Y\vee \F_{k-\lfloor l/2\rfloor+1,\infty}^\eps$-measurable. Clearly, $\E (|\Psi_1|^{2/\epsilon+1})=\E (|\phi (\th_{k-l})|^{2/\epsilon+1})$, and by the stationarity of $((Y_k,\eps_{k+1}))_{k\in\Z}$, $\E (|\Psi_2|^{2/\epsilon+1}) = \E (|\phi (\th_{\lfloor l/2\rfloor})|^{2/\epsilon+1})$ hence by Lemma \ref{lem:UI}, we have
$$
\E (|\Psi_i|^{2/\epsilon+1})\le \tilde{c}_{2/\epsilon+1}:= c_{2/\epsilon+1}^{2/\epsilon+1},\,\,i=1,2.
$$	

So we can apply Lemma \ref{lem:ibragimov} for $\Psi_1$ and $\Psi_2$, and by Remark \ref{rem:mix}, we obtain
\begin{align}\label{eq:cest1}
\begin{split}
\left|\cov \left(
\phi(Z_{k-\lfloor l/2\rfloor,k}^{\th_0,\mathbf{Y}}),
\phi(Z_{0,k-l}^{\th_0,\mathbf{Y}})\right)\right|
&\le 
(4+5\tilde{c}_{2/\epsilon+1}) \alpha (\F_{-\infty, k-l-1}^Y\vee \F_{-\infty,k-l}^\eps,\F_{k-\lfloor l/2\rfloor,\infty}^Y\vee \F_{k-\lfloor l/2\rfloor+1,\infty}^\eps)^{1-\epsilon}
\\
&=
(4+5\tilde{c}_{2/\epsilon+1})
\alpha (\F_{-\infty, k-l-1}^Y,\F_{k-\lfloor l/2\rfloor,\infty}^Y)^{1-\epsilon}
\\
&= (4+5\tilde{c}_{2/\epsilon+1}) \alpha^Y (l+1-\lfloor l/2\rfloor)^{1-\epsilon}
\le (4+5\tilde{c}_{2/\epsilon+1}) \alpha^Y (\lfloor l/2\rfloor)^{1-\epsilon}.
\end{split}
\end{align}

In the second term on the right-hand side of \eqref{eq:terms}, by the Cauchy--Schwartz inequality and Lemma \ref{lem:UI}, we have
\begin{align*}
\left|\cov \left(
\phi(Z_{0,k}^{\th_0,\mathbf{Y}})-\phi(Z_{k-\lfloor l/2\rfloor,k}^{\th_0,\mathbf{Y}}),
\phi(Z_{0,k-l}^{\th_0,\mathbf{Y}})\right)\right|
&\le 
c_2 \E^{1/2}\left(\ind_{Z_{0,k}^{\th_0,\mathbf{Y}}\ne Z_{k-\lfloor l/2\rfloor,k}^{\th_0,\mathbf{Y}} }\left(\phi(Z_{0,k}^{\th_0,\mathbf{Y}})-\phi(Z_{k-\lfloor l/2\rfloor,k}^{\th_0,\mathbf{Y}})\right)^2\right)
\\
&\le
2 c_2 c_4 \P (Z_{0,k}^{\th_0,\mathbf{Y}}\ne Z_{k-\lfloor l/2\rfloor,k}^{\th_0,\mathbf{Y}})^{1/4}
\end{align*}
Note that $2N\le l$, and thus by Lemma \ref{lem:annealed} and Corollary \ref{cor:drift}, we can write
\begin{align*}
\P \left(Z_{0,k}^{\th_0,\mathbf{Y}}\ne Z_{k-\lfloor l/2\rfloor,k}^{\th_0,\mathbf{Y}}\right)
&=
\P \left(Z_{k-\lfloor l/2\rfloor,k}^{\th_{k-\lfloor l/2\rfloor},\mathbf{Y}}\ne Z_{k-\lfloor l/2\rfloor,k}^{\th_0,\mathbf{Y}}\right)
\le
\left(V(\th_0)+\frac{3}{2}+\frac{C}{2(1-\ga)}\right)
e^{-\ka \lfloor l/2\rfloor},
\end{align*}
and thus we arrive at
\begin{equation}\label{eq:cest2}
\left|\cov \left(
\phi(Z_{0,k}^{\th_0,\mathbf{Y}})-\phi(Z_{k-\lfloor l/2\rfloor,k}^{\th_0,\mathbf{Y}}),
\phi(Z_{0,k-l}^{\th_0,\mathbf{Y}})\right)\right|
\le 
2 c_2 c_4
\left(V(\th_0)+\frac{3}{2}+\frac{C}{2(1-\ga)}\right)^{1/4}
e^{-\frac{\ka}{4} \lfloor l/2\rfloor}.
\end{equation}
Substituting \eqref{eq:cest1} and \eqref{eq:cest2} into \eqref{eq:terms}, and replacing $k$ by $k+l$ yields 
\begin{equation*}
|\cov (\phi (\th_k),\phi (\th_{k+l}))|\le 
(4+5\tilde{c}_{2/\epsilon+1}) \alpha^Y (\lfloor l/2\rfloor)^{1-\epsilon}
+
2 c_2 c_4\left(V(\th_0)+\frac{3}{2}+\frac{C}{2(1-\ga)}\right)^{1/4}
e^{-\frac{\ka}{4} \lfloor l/2\rfloor}
\end{equation*}
which completes the proof.

\end{proof}

\bibliography{clt_langevin_miklos}
\bibliographystyle{plain}
	
\end{document}